\documentclass[12pt, final]{amsart}
\usepackage[english]{babel}
\usepackage{latexsym}

\usepackage{amsmath}
\usepackage{eucal}

\newcommand\restr[2]{{\left.\kern-\nulldelimiterspace
  #1
  \vphantom{\big|}
  \right|_{#2}
  }}

\DeclareMathOperator{\ad}{ad}

\DeclareMathOperator{\Id}{id}

\DeclareMathOperator{\End}{End}

\newcommand{\F}{\mathbb F}

\newcommand{\LL}{\mathcal L}

\newtheorem{dummy}{Dummy}

\numberwithin{dummy}{section}
\numberwithin{equation}{section}

\newtheorem{lemma}[dummy]{Lemma}

\newtheorem{theorem}[dummy]{Theorem}

\newtheorem{prop}[dummy]{Proposition}

\theoremstyle{definition}

\theoremstyle{remark}
\newtheorem{rem}[dummy]{Remark}

\begin{document}

\bibliographystyle{amsalpha}

\author{Marina Avitabile}
\email{marina.avitabile@unimib.it}
\address{Dipartimento di Matematica e Applicazioni\\
  Universit\`a degli Studi di Milano - Bicocca\\
  via Cozzi 53\\
  I-20125 Milano\\
  Italy}

\author{Sandro Mattarei}
\email{mattarei@science.unitn.it}
\address{Dipartimento di Matematica\\
  Universit\`a degli Studi di Trento\\
  via Sommarive 14\\
  I-38050 Povo (Trento)\\
  Italy}

\title[Laguerre polynomials of derivations]{Laguerre polynomials of derivations}

\begin{abstract}
We introduce a {\em grading switching} for arbitrary nonassociative
algebras of prime characteristic $p$, aimed at producing a new grading of an algebra from a given one.
We take inspiration from a fundamental tool in the classification theory of modular Lie algebras known as {\em toral
switching,} which relies on a delicate adaptation of the exponential of a
derivation.
Our grading switching is achieved by evaluating certain generalized Laguerre polynomials of degree
$p-1$, which play the role of generalized exponentials,
on a derivation of the algebra.
A crucial part of our argument is establishing a congruence
for them which is an appropriate analogue of the functional equation
$e^x\cdot e^y=e^{x+y}$ for the classical exponential.
Besides having a wider scope, our treatment provides a more transparent explanation of some
aspects of the original toral switching, which can be recovered as a special
case.
\end{abstract}
\maketitle

\section{Introduction}

The exponential function is certainly one of the
most important mathematical functions.
The main reason, sometimes disguised in other forms, such as its differential
formulation
$(d/dx)e^x=e^x$, is that it
interconnects additive and multiplicative structures,
because of the fundamental identity $e^x\cdot e^y=e^{x+y}$.
In particular, one of the important classical applications is the local
reconstruction of a Lie group from its Lie algebra.
This Lie-theoretic use of the exponential function can be
formulated in more general terms as a device which turns derivations
of a {\em nonassociative} (in the standard meaning of {\em not necessarily associative}) algebra into automorphisms.
The basic algebraic fact is already visible in the special case of nilpotent derivations,
where convergence matters play no role:
if $D$ is a nilpotent derivation of a non associative algebra $A$ over a field of characteristic zero,
then the finite sum $\exp(D)=\sum_{i=0}^{\infty}D^i/i!$ defines an
automorphism of $A$.

This nice property breaks down over fields of positive
characteristic $p$.
The condition $D^p=0$, which seems the minimum requirement for $\exp(D)$ to make sense in
this context, does not guarantee that $\exp(D)$ is an automorphism.
In fact, only the stronger assumption $D^{(p+1)/2}=0$ does, for
$p$ odd.
In the absence of the assumption $D^p=0$ one can use the {\em truncated exponential}
$E(X)=\sum_{i=0}^{p-1}X^i/i!$ as some kind of substitute for the
exponential series, of course dropping any expectation
that evaluating it on $D$ may yield an automorphism.

In the theory of modular Lie algebras the apparent shortcoming of $\exp(D)$
not necessarily being an automorphism when it is defined is
turned into an advantage with the technique of {\em toral
switching}.
This is a fundamental tool which originated in~\cite{Win:toral},
but has undergone substantial generalizations in~\cite{BlWil:rank-two} and finally~\cite{Premet:Cartan},
where maps similar to exponentials of derivations are used to
produce a new torus from a given one.
The very fact that the map need not not be an automorphism
allows the new torus to have rather different
properties than the original one, which are more suited to classification purposes.

A crucial function of tori in modular Lie algebras
is to produce gradings, as the corresponding eigenspace
decompositions with respect to the adjoint action (a {\em (generalized) root space decomposition}).
One naturally wonders whether some kind of exponential could be used
to pass from a grading to another without reference to the grading
arising as the root space decomposition with respect to some
torus.
Besides effectively extending the applicability of the technique
from the realm of Lie algebras to the wider one of nonassociative
algebras, such {\em grading switching} does have applications
within Lie algebra theory, where not all gradings of interest
are directly related to tori.
A special instance of such grading switching was described
in~\cite{Mat:Artin-Hasse}, in terms of {\em Artin-Hasse
exponentials}.
The strong limitation of~\cite{Mat:Artin-Hasse} was that the
derivation $D$ had to be nilpotent, but that
special version was already sufficient for an
application to certain Lie algebra gradings in~\cite{AviMat:-1}.

The main goal of the present work is to describe how grading
switching, in the spirit of~\cite{Mat:Artin-Hasse},
can be done in full generality,
for arbitrary derivations of nonassociative Lie algebras
(with respect to compatible gradings).
We will show how this extends the classical toral switching in a natural way.
The role of the exponential series is taken by certain
(generalized) {\em Laguerre polynomials}, which suggest the title of this paper.

We only sketch the essence of our main result in this introduction
and refer the reader to Section~\ref{sec:general_case} for
a precise formulation.
Let $A=\bigoplus A_k$ be a nonassociative algebra over a field
$\F$ of prime characteristic $p$, graded over the integers modulo $m$,
and let $D$ be a graded derivation of $A$, whose degree $d$
satisfies $m\mid pd$.
Assume $\F$ algebraically closed
and $A$ finite-dimensional for simplicity, but much weaker assumptions
are sufficient and will be specified later.
Then we construct a linear map
$\LL_D:A\to A$ such that
$A=\bigoplus_k \LL_D(A_k)$ is a new grading over the integers modulo $m$.
In the special case where $D$ is nilpotent the map $\LL_D$
coincides with $E_p(D)$, which denotes the Artin-Hasse
exponential series evaluated on $D$ and was investigated in~\cite{Mat:Artin-Hasse}.

In Section~\ref{sec:special} we prove a special case of our
result, where the main argument is stripped of the distraction of
some additional technicalities of the general case.

Both the special case and the general case depend on a congruence
for certain Laguerre polynomials, which we prove in
Section~\ref{sec:exponential} and which might well be of interest
outside the area of nonassociative algebras.
It is a polynomial congruence analogue of the
functional equation $\exp(X)\exp(Y)=\exp(X+Y)$ for the classical
exponential.

Section~\ref{sec:Laguerre} contains a review of definitions and
known properties of Laguerre polynomials, and also a modular
property which might be new.

In the concluding Section~\ref{sec:toral_switching} we explain how our
main result specializes to the setting of toral switching in
finite-dimensional modular Lie algebras.

Our preprint~\cite{AviMat:mixed_types}
contains a result on certain modular Lie algebras
whose proof depends on a grading switching as described here.

\section{Laguerre polynomials and some of their properties}\label{sec:Laguerre}

The classical (generalized) Laguerre polynomial of degree $n \geq 0$ is defined as
\[
L_n^{(\alpha)}(X)=\sum_{k=0}^n\binom{\alpha+n}{n-k}
\frac{(-X)^k}{k!},
\]
where $\alpha$ is a parameter, usually taken in the complex numbers.
However, we may also view $L_n^{(\alpha)}(X)$ as a polynomial with rational coefficients in the two
indeterminates $\alpha$ and $X$.
It is well known and easy to check that the Laguerre polynomials satisfy the identities
\begin{align}
&L_n^{(\gamma)}(X)= L_n^{(\gamma+1)}(X)-L_{n-1}^{(\gamma+1)}(X)\label{eq:rel1},
\\
&nL_n^{(\gamma+1)}(X)= (n-X)L_{n-1}^{(\gamma+1)}(X)+(n+\gamma)L_{n-1}^{(\gamma)}(X) \label{eq:rel2}.
\end{align}
The derivative of $L_n^{(\gamma)}(X)$ with respect to $X$
equals $-L_{n-1}^{(\gamma+1)}(X)$, which according to
Equation~\eqref{eq:rel1} can be written as
\begin{equation}\label{eq:derivative}
\frac{d}{dX}
L_n^{(\gamma)}(X)
=
L_n^{(\gamma)}(X)-L_n^{(\gamma+1)}(X),
\end{equation}

Now fix a prime $p$.
We are essentially interested only in the polynomial
$L_{p-1}^{(\alpha)}(X)$.
The reason is that, viewed in characteristic $p$, it may be thought of as a generalization
of the truncated exponential
$E(X)=\sum_{k=0}^{p-1}X^k/k!$
which we mentioned in the introduction.
In fact, we have
$L_{p-1}^{(0)}(X)\equiv E(X)\pmod{p}$
because
$\binom{p-1}{k}\equiv\binom{-1}{k}=(-1)^k$
for $k\ge 0$,
and the full sense of this generalization should be conveyed by
the easily verified congruence
\begin{equation}\label{eq:Lmodp}
L_{p-1}^{(\alpha)}(X)
\equiv
(1-\alpha^{p-1})
\sum_{k=0}^{p-1}\frac{X^k}
{(\alpha+k)(\alpha+k-1)\cdots(\alpha+1)}
\pmod{p}.
\end{equation}

In this preparatory section we collect some properties of
$L_{p-1}^{(\alpha)}(X)\pmod{p}$,
starting with some easy ones.
Equation \eqref{eq:rel2} with $n=p$ yields
\[
pL_p^{(\gamma+1)}(X)=(p-X)L_{p-1}^{(\gamma+1)}(X)+(p+\gamma)L_{p-1}^{(\gamma)}(X).
\]
Because
\[
pL_p^{(\gamma+1)}(X)=p\sum_{k=0}^p\binom{\gamma+1+p}{p-k}\frac{(-X)^k}{k!}
\equiv X^p-(\gamma^p-\gamma)\pmod{p},
\]
we deduce the congruence
\begin{equation}\label{eq:pLp}
X^p-(\gamma^p-\gamma)\equiv -XL_{p-1}^{(\gamma+1)}(X)+ \gamma L_{p-1}^{(\gamma)}(X)
\pmod{p}.
\end{equation}
Equation~\eqref{eq:pLp} allows one to give
Equation~\eqref{eq:derivative} for the derivative of $L_{p-1}^{(\gamma)}(X)$
a variant in congruence form which we will use later, namely,
\begin{equation}\label{eq:L-diff}
X\cdot\frac{d}{dX}
L_{p-1}^{(\gamma)}(X)
\equiv
(X-\gamma)\cdot L_{p-1}^{(\gamma)}(X)
+
X^p-(\gamma^p-\gamma)
\pmod{p}.
\end{equation}
In the special case where $\gamma=0$ this reads
\begin{equation}\label{eq:E-diff}
XE'(X)
\equiv
XE(X)+X^p
\pmod{p}
\end{equation}
in terms of the truncated exponential $E(X)$.
Because of this analogy with the defining differential equation
$\exp'(X)=\exp(X)$
for the classical exponential,
Equation~\eqref{eq:L-diff} plays a key role in the proof of our
Proposition~\ref{prop:equation_for_L},
which, in turn, is crucial for our main result.
Note in passing that further differentiation of
Equation~\eqref{eq:E-diff} leads to
$XE''(X)+(1-X)E'(X)-E(X)\equiv 0\pmod{p}$.
This is a special case modulo $p$ of the second-order
differential equation
$XY''+(\alpha+1-X)Y'+nY=0$,
which is often used to define the Laguerre polynomials
$Y=L_n^{(\alpha)}(X)$.

Now we present a property of $L_{p-1}^{(\alpha)}(X)\pmod{p}$
which appears more hidden, and might well be of more general interest.
To avoid constant use of the `mod $p$' notation, in the remainder of the paper
the Laguerre polynomial $L_{p-1}^{(\alpha)}(X)$ will always be viewed as having
coefficients in $\F_p$, the field with $p$ elements.
The polynomial
$L_{p-1}^{(Z^p)}(Z^p-Z)$
will play a special role in the sequel.
Equation~\eqref{eq:Lmodp} shows at once that it
vanishes on $\F_p^{\ast}$,
but what we will actually need later is that it has no further
roots in in the algebraic closure
$\overline{\F_p}$ of $\F_p$.
This is a consequence of the following result.

\begin{lemma}\label{lemma:Laguerre}
We have
$L_{p-1}^{(Z^p)}(Z^p-Z)=\prod_{i=1}^{p-1}(1+Z/i)^i$
in $\F_p[Z]$.
\end{lemma}

This can also be stated in the equivalent
form
\[
L_{p-1}^{(Z^p)}(Z^p-Z)
=(-1)^{p(p-1)/2}\prod_{j=1}^{p-1}\binom{Z-1}{j}
\]
in $\F_p[Z]$.
In fact, the right-hand sides of the two equations are polynomials with the same
roots in
$\overline{\F_p}$, with corresponding multiplicities, and the same constant
term $1$ because
$\prod_{j=1}^{p-1}\binom{-1}{j}
=\prod_{j=1}^{p-1}(-1)^j
=(-1)^{p(p-1)/2}$.
A nontrivial consequence of Lemma~\ref{lemma:Laguerre} is the fact that
$L_{p-1}^{(Z^p)}(Z^p-Z)$ has degree $p(p-1)/2$.
Another noteworthy consequence is the identity
\begin{equation*}
L_{p-1}^{(Z^p)}(Z^p-Z)
\cdot
L_{p-1}^{(-Z^p)}(-Z^p+Z)
=
1-Z^{p(p-1)}.
\end{equation*}
in $\F_p[Z]$,
to be compared with the familiar $\exp(X)\exp(-X)=1$,
after noting that $L_{p-1}^{(Z^p)}(Z^p-Z)\equiv L_{p-1}^{(0)}(-Z)=E(-Z)\pmod{Z^p}$.

\begin{proof}[Proof of Lemma~\ref{lemma:Laguerre}]
Equation~\eqref{eq:pLp} yields
\begin{equation}\label{eq:shift}
(Z^p-Z)\cdot L_{p-1}^{(Z^p+1)}(Z^p-Z)
=Z^p\cdot L_{p-1}^{(Z^p)}(Z^p-Z)
\end{equation}
in $\F_p[Z]$.
Note that $L_{p-1}^{(0)}(0)=1$
and write
$L_{p-1}^{(Z^p)}(Z^p-Z)=\prod_{i=1}^{s}(1-Z/\alpha_i)$
in $\overline{\F_p}[Z]$.
Then Equation~\eqref{eq:shift} says that
\[
\prod_{j=1}^{p-1}(Z-j)\cdot \prod_{i=1}^{s}\bigl(Z-(\alpha_i-1)\bigr)
=Z^{p-1}\cdot \prod_{i=1}^{s}(Z-\alpha_i).
\]
We infer that if some $\alpha\in\overline{\F_p}$ is a root of
$L_{p-1}^{(Z^p)}(Z^p-Z)$ with multiplicity $m$, where we allow $m$ to be zero, then $\alpha+1$
is a root with multiplicity $m+p-1$ if $\alpha=0$, $m-1$ if $\alpha\in\F_p^\ast$,
and $m$ otherwise.
In particular, because $0$ is not a root, each element of $\F_p$
is a root of $L_{p-1}^{(Z^p)}(Z^p-Z)$ with the multiplicity claimed in Lemma~\ref{lemma:Laguerre}.

Because $L_{p-1}^{(Z^p)}(Z^p-Z)$ has constant term $1$,
in order to conclude the proof it remains to show that it has no further roots in $\overline{\F_p}$.
To prove that it suffices to show that the polynomial has degree at most
$p(p-1)/2$.
Note that direct expansion only shows us that
it has degree at most $p(p-1)$, which is twice as high as our
goal.
One way to proceed is noting that according to Equation~\eqref{eq:shift} the product
\begin{equation}\label{eq:product}
Z^p\cdot L_{p-1}^{(Z^p)}(Z^p-Z)
\cdot
L_{p-1}^{(-Z^p)}(-Z^p+Z)
\end{equation}
is invariant under the substitution $Z\mapsto Z+1$, and hence can be expressed as a polynomial in $Z^p-Z$.
However, because its derivative is zero, as we prove in the next paragraph, it can also be expressed
as a polynomial in $Z^p$.
These conditions together imply that it can be expressed as a polynomial in
$Z^{p^2}-Z^p$.
Because we know that its degree cannot exceed $2p^2-p$ we infer that
it cannot exceed $p^2$, whence $L_{p-1}^{(Z^p)}(Z^p-Z)$
has degree at most $p(p-1)/2$, as desired.

Now we prove our claim about the polynomial of Equation~\eqref{eq:product}
having zero derivative.
According to Equation~\eqref{eq:derivative} we have
\[
\frac{d}{dZ}L_{p-1}^{(Z^p)}(Z^p-Z)=-L_{p-1}^{(Z^p)}(Z^p-Z)+L_{p-1}^{(Z^p+1)}(Z^p-Z),
\]
and hence the derivative of that polynomial
equals the product of $Z^p$ and
\begin{align*}
\frac{d}{dZ}
&\bigl(L_{p-1}^{(Z^p)}(Z^p-Z)
\cdot
L_{p-1}^{(-Z^p)}(-Z^p+Z)\bigr)
\\&=
L_{p-1}^{(Z^p+1)}(Z^p-Z)\cdot L_{p-1}^{(-Z^p)}(-Z^p+Z)
-
L_{p-1}^{(Z^p)}(Z^p-Z)\cdot L_{p-1}^{(-Z^p+1)}(-Z^p+Z)
\\&=
\left(\frac{Z^p}{Z^p-Z}-\frac{(-Z)^p}{(-Z)^p-(-Z)}\right)\cdot
L_{p-1}^{(Z^p)}(Z^p-Z)\cdot L_{p-1}^{(-Z^p)}(-Z^p+Z)
=0,
\end{align*}
where in the last step we have used Equation~\eqref{eq:shift}
twice, with $-Z$ in place of $Z$ in the latter case.
\end{proof}

\section{An exponential-like property of $L_{p-1}^{(\alpha)}(X)$}\label{sec:exponential}

In this section we use the differential equation modulo $p$ for $L_{p-1}^{(\alpha)}(X)$
which we stated in Equation~\eqref{eq:L-diff} to prove a congruence similar to the
functional equation $\exp(X)\exp(Y)=\exp(X+Y)$ satisfied by the classical
exponential.

\begin{prop}\label{prop:equation_for_L}
Consider the subring
$R=\F_p\bigl[\alpha,\beta,\bigl((\alpha+\beta)^{p-1}-1\bigr)^{-1}\bigr]$ of the ring
$\F_p(\alpha,\beta)$ of rational expressions in the indeterminates $\alpha$ and $\beta$,
and let $X$ and $Y$ be further indeterminates.
There exist rational expressions $c_i(\alpha,\beta)\in R$, such that
\[
L_{p-1}^{(\alpha)}(X)
L_{p-1}^{(\beta)}(Y)
\equiv
L_{p-1}^{(\alpha+\beta)}(X+Y)
\Bigl(
c_0(\alpha,\beta)+\sum_{i=1}^{p-1}c_i(\alpha,\beta)X^iY^{p-i}
\Bigr)
\]
in $R[X,Y]$, modulo the ideal generated by
$X^p-(\alpha^p-\alpha)$
and
$Y^p-(\beta^p-\beta)$.
\end{prop}

The crucial point for our applications of
Proposition~\ref{prop:equation_for_L} is that the polynomial
$c_0(\alpha,\beta)+\sum_{i=1}^{p-1}c_i(\alpha,\beta)X^iY^{p-i}$
has only terms of total degree a multiple of $p$.
A simplified but weaker form of Proposition~\ref{prop:equation_for_L}
is that the stated congruence holds in $\F_p(\alpha,\beta)[X,Y]$, modulo the ideal
generated by the stated elements.
This weaker statement would suffice for the proof of
Theorem~\ref{thm:special}, but not for that of
Theorem~\ref{thm:general}.

\begin{proof}
Let $\mathcal{I}$ denote the ideal of the polynomial ring $R[X,Y]$ generated by
$X^p-(\alpha^p-\alpha)$
and
$Y^p-(\beta^p-\beta)$.
According to Lemma~\ref{lemma:Laguerre} we have
\begin{equation}\label{eq:L_as_product}
\begin{aligned}
\bigl(L_{p-1}^{(\alpha+\beta)}(X+Y)\bigr)^p
&=
L_{p-1}^{((\alpha+\beta)^p)}((X+Y)^p)
\\&\equiv
L_{p-1}^{((\alpha+\beta)^p)}((\alpha+\beta)^p-(\alpha+\beta))
\pmod{\mathcal{I}}
\\&=
\prod_{i=1}^{p-1}\Bigl(1+\frac{\alpha+\beta}{i}\Bigr)^i,
\end{aligned}
\end{equation}
a nonzero element of $R$.
In particular, the image of $L_{p-1}^{(\alpha+\beta)}(X+Y)$
in the quotient ring $R[X,Y]/\mathcal{I}$ is invertible.
Reading congruences as equalities in the corresponding quotient
ring, we have
\begin{equation}\label{eq:L-quotient}
\frac{
L_{p-1}^{(\alpha)}(X)
L_{p-1}^{(\beta)}(Y)
}{
L_{p-1}^{(\alpha+\beta)}(X+Y)
}
\equiv
\sum_{i,j=0}^{p-1}c'_{ij}(\alpha,\beta)X^iY^j
\pmod{\mathcal{I}},
\end{equation}
for certain (uniquely determined)
$c'_{ij}(\alpha,\beta)\in R$.
Our goal is proving that
$c'_{ij}(\alpha,\beta)$ vanishes when $p$ does not divide $i+j$.

Following~\cite{Mat:exponential} we introduce a further indeterminate $T$ and consider
the polynomial ring $R[X,Y,T]$
and its ideal $\mathcal{I}_T$ generated by
$(TX)^p-(\alpha^p-\alpha)$
and
$(TY)^p-(\beta^p-\beta)$.
The epimorphism $R[X,Y,T]=R[X,Y][T]$ onto $R[X,Y]$
given by evaluation at $T=1$
maps $\mathcal{I}_T$ onto $\mathcal{I}$, and hence induces an
epimorphism of $R[X,Y,T]/\mathcal{I}_T$ onto
$R[X,Y]/\mathcal{I}$.
Substituting $TX$ for $X$ and $TY$ for $Y$ in
Equation~\eqref{eq:L-quotient} yields
\begin{equation*}
\frac{
L_{p-1}^{(\alpha)}(TX)
L_{p-1}^{(\beta)}(TY)
}{
L_{p-1}^{(\alpha+\beta)}(TX+TY)
}
\equiv
\sum_{i,j=0}^{p-1}c'_{ij}(\alpha,\beta)T^{i+j}X^iY^j
\pmod{\mathcal{I}_T}.
\end{equation*}
Because the differential operator $d/dT$ on
$R[X,Y,T]$, with kernel $R[X,Y,T^p]$, maps the ideal $\mathcal{I}_T$ into itself, it induces a derivation of the quotient ring
$R[X,Y,T]/\mathcal{I}_T$.
Hence proving that
$c'_{ij}(\alpha,\beta)$ vanishes when $p$ does not divide $i+j$
is equivalent to proving that
\begin{equation}\label{eq:T}
\frac{d}{dT}\frac{
L_{p-1}^{(\alpha)}(TX)
L_{p-1}^{(\beta)}(TY)
}{
L_{p-1}^{(\alpha+\beta)}(TX+TY)
}\equiv 0
\pmod{\mathcal{I}_T}
\end{equation}
in $R[X,Y,T]$.
After expanding via Leibnitz's rule and evaluating at $T=1$
(which can be reversed by substituting $TX$ for $X$ and $TY$ for $Y$)
we see that Equation~\eqref{eq:T} is equivalent to
\begin{multline}\label{eq:log_der}
X
L_{p-1}^{(\alpha)}(X)'\cdot L_{p-1}^{(\beta)}(Y)
+
L_{p-1}^{(\alpha)}(X)\cdot Y L_{p-1}^{(\beta)}(Y)'
\\
\equiv
(X+Y)\frac{
L_{p-1}^{(\alpha+\beta)}(X+Y)'
}
{
L_{p-1}^{(\alpha+\beta)}(X+Y)
}
\cdot
L_{p-1}^{(\alpha)}(X)\cdot L_{p-1}^{(\beta)}(Y)
\pmod{\mathcal{I}}
\end{multline}
where we have used the shorthand $L_n^{(\gamma)}(Z)'=(d/dZ)L_n^{(\gamma)}(Z)$.
According to Equation~\eqref{eq:L-diff} we have
\[
Z
L_{p-1}^{(\gamma)}(Z)'
=
(Z-\gamma)L_{p-1}^{(\gamma)}(Z)
+
Z^p-(\gamma^p-\gamma).
\]
Taking, in turn, $Z=X$, $Z=Y$, and $Z=X+Y$, shows that Equation~\eqref{eq:log_der} holds.
In conclusion, we have proved that
\begin{equation}\label{eq:L-series''}
\frac{
L_{p-1}^{(\alpha)}(X)
L_{p-1}^{(\beta)}(Y)
}{
L_{p-1}^{(\alpha+\beta)}(X+Y)
}
\equiv
c_0(\alpha,\beta)+\sum_{i=1}^{p-1}c_i(\alpha,\beta)X^iY^{p-i}
\pmod{\mathcal{I}}
\end{equation}
with $c_0(\alpha,\beta):=c'_{0,0}(\alpha,\beta)$, and $c_i(\alpha,\beta):=c'_{i,p-i}(\alpha,\beta)$ for $0<i<p$.
\end{proof}

The specialization of Proposition~\ref{prop:equation_for_L} to $\alpha=\beta=0$,
where
$L_{p-1}^{(0)}(X)$
equals the {\em truncated exponential}
$E(X)=\sum_{i=0}^{p-1}X^i/i!$,
takes the more precise form
\[
E(X)\cdot E(Y)\equiv
E(X+Y)
\Bigl(1+\sum_{i=1}^{p-1}(-1)^iX^iY^{p-i}/i\Bigr)
\]
in $\F_p[X,Y]$, modulo the ideal generated by $X^p$ and $Y^p$,
see~\cite[Lemma~2.1]{Mat:Artin-Hasse}.
This can be viewed as a truncated version of a corresponding
property of (the reduction modulo $p$ of) the
Artin-Hasse exponential series, which is defined as
\[
E_p(X):=
\exp\Bigl(\sum_{i=0}^{\infty}X^{p^i}/p^i\Bigr)=
\prod_{i=0}^{\infty}\exp(X^{p^i}/p^i).
\]
In fact, as shown in the proof
of~\cite[Theorem~2.2]{Mat:Artin-Hasse},
we have
\[
E_p(X)\cdot E_p(Y)=
E_p(X+Y)\Bigl(1+\sum_{i,j=1}^{\infty}a_{ij}X^iY^j\Bigr)
\]
in $\F_p[[X,Y]]$, for certain coefficients $a_{ij}\in\F_p$
which vanish unless $p\mid i+j$.
It was proved in~\cite{Mat:exponential} that this property
essentially characterizes the reduction modulo $p$ of the
Artin-Hasse series, up to some natural variations.

\section{A model special case}\label{sec:special}

In order to avoid that too many technical details may obscure our main argument, we first present
an application of Proposition~\ref{prop:equation_for_L} to a special situation,
and postpone consideration of a more general setting to the next section.

\begin{theorem}\label{thm:special}
Let $A=\bigoplus_k A_k$ be a nonassociative algebra over the field $\F$ of characteristic $p>0$,
graded over the integers modulo $m$.
Suppose that $A$ has a graded derivation $D$ of degree $d$
such that  $D^{p^2}=D^p$, with $m\mid pd$.
Suppose that $\F$ contains the field of $p^p$ elements, and choose $\gamma \in \F$ with
$\gamma^p-\gamma=1$.
Let $A=\bigoplus_{a\in\F_p}A^{(a)}$ be the decomposition of $A$
into a direct sum of generalized eigenspaces for $D$,
and let $\LL_D:A\to A$ be the linear map whose restriction to
$A^{(a)}$coincides with $L_{p-1}^{(a\gamma)}(D)$.
Then
$A=\bigoplus_k\LL_D(A_k)$ is also a grading of $A$ over the integers modulo $m$.
\end{theorem}

\begin{proof}
In this special case the eigenvalues of $D^p$ are elements of the prime field $\F_p$,
hence of the form $\alpha^p-\alpha$, with $\alpha=a\gamma$ for
some $a\in\F_p$.

The linear map $\LL_D$ is bijective.
In fact, $(\LL_D)^p$ acts
on the eigenspace $A^{(a)}$ of $D^p$
as multiplication by the scalar
\[
\bigl(L_{p-1}^{(a\gamma)}(a)\bigr)^p
=
L_{p-1}^{((a\gamma)^p)}\bigl((a\gamma)^p-a\gamma\bigr),
\]
which is nonzero according to Lemma~\ref{lemma:Laguerre}.
Hence we have the direct sum decomposition $A=\bigoplus_k\LL_D(A_k)$.

In order to prove that this is a grading
we need to  prepare the ground for an application of Proposition~\ref{prop:equation_for_L},
in a similar way as was done in the proof of~\cite[Theorem~2.2]{Mat:Artin-Hasse} in case of the Artin-Hasse exponential.
If $m: A\otimes A \rightarrow A$ denotes the map given
by the multiplication in $A$,
the fact that $D$ is a derivation means that
$D(m(x \otimes y))=m(Dx \otimes y)+m(x \otimes Dy)$ for any $x,y \in A$.
This property can be more concisely written as
$D\circ m=m \circ(D\otimes\Id+\Id\otimes D)$, where $\Id:A \rightarrow A$ is the
identity map.
In particular, we have
$L_{p-1}^{(\theta)}(D)\circ m=m \circ L_{p-1}^{(\theta)}(D\otimes\Id+\Id\otimes D)$
for any $\theta\in\F$.
The multiplication map $m$ restricts to a map
$A^{(a)}\otimes A^{(b)}\rightarrow A^{(a+b)}$,
for any $a,b\in\F_p$,
and all the components involved are invariant under $D$.
Viewing the commuting linear operators $D\otimes\Id$ and $\Id\otimes D$ as restricted to
$A^{(a)}\otimes A^{(b)}$, the congruence of Proposition~\ref{prop:equation_for_L}
can be evaluated on $D\otimes\Id$ and $\Id\otimes D$ for $X$ and $Y$,
with $a\gamma$ and $b\gamma$ for $\alpha$ and $\beta$.
This is because $D^p$ acts as multiplication by
$a^p=a=(a\gamma)^p-a\gamma$ on $A^{(a)}$, and similarly for
$A^{(b)}$.
Also note that the rational expressions $c_i(\alpha,\beta)$ can be
evaluated on $a\gamma$ and $b\gamma$ because
$(a+b)\gamma\not\in\F_p^\ast$, which is equivalent to
$\bigl((a+b)\gamma\bigr)^{p-1}\not=1$.
The result of this evaluation, followed by composition with the multiplication map $m$, is that the restriction of
$m\circ
\bigl(L_{p-1}^{(a\gamma)}(D)
\otimes
L_{p-1}^{(b\gamma)}(D)\bigr)$
to
$A^{(a)}\otimes A^{(b)}$
coincides with the restriction of
\[
L_{p-1}^{(a\gamma+b\gamma)}(D)
\circ m\circ
\Bigl(
c_0(a\gamma,b\gamma)+\sum_{i=1}^{p-1}c_i(a\gamma,b\gamma)(D^i\otimes D^{p-i})
\Bigr).
\]
This means that
\[
L_{p-1}^{(a\gamma)}(D)x \cdot L_{p-1}^{(b\gamma)}(D)y
=
L_{p-1}^{(a\gamma+b\gamma)}(D)
\Bigl(c_0(a\gamma,b\gamma)xy+\sum_{i=1}^{p-1}c_i(a\gamma,b\gamma)D^ix\cdot D^{p-i}y\Bigr)
\]
for $x\in A^{(a)}$ and $y\in A^{(b)}$, which we can also write as
\begin{equation}\label{eq:L_D}
\LL_Dx \cdot \LL_Dy
=
\LL_D
\Bigl(c_0(a\gamma,b\gamma)xy+\sum_{i=1}^{p-1}c_i(a\gamma,b\gamma)D^ix\cdot
D^{p-i}y\Bigr).
\end{equation}

Because $D^p$ is a graded derivation of degree zero, it maps each
component $A_k$ of the grading into itself, and hence
$A_k=\bigoplus_{a\in\F_p}A_k\cap A^{(a)}$.
Because $m\mid pd$, the term $D^ix\cdot D^{p-i}y$ in
Equation~\eqref{eq:L_D}, for $x\in A_k\cap A^{(a)}$
and $y\in A_\ell\cap A^{(b)}$,
belongs to $A_{k+\ell}$ as well as the term $xy$.
Hence Equation~\eqref{eq:L_D} implies that
$\LL_Dx\cdot \LL_Dy\in \LL_D\bigl(A_{k+\ell}\cap A^{(a+b)}\bigr)$.
In particular, we conclude that
$\LL_D A_k\cdot \LL_D A_\ell\subseteq \LL_D A_{k+\ell}$,
and so
$A=\bigoplus_k \LL_D(A_k)$
is a grading of $A$ over the integers modulo $m$.
\end{proof}

\section{The general case}\label{sec:general_case}

In this section we prove our main result, which extends Theorem~\ref{thm:special} to the general case
where $D$ is a derivation of a nonassociative algebra $A$ over a field
$\F$ of characteristic $p$, which we assume as large as we need in this paragraph,
under the sole assumption on $D$
that $D^{p^r}$ is semisimple with finitely many
eigenvalues, for some $r$.
In fact, in that case $D$ satisfies an equation
\begin{equation}\label{eq:D}
D^{p^n}+a_{n-1}D^{p^{n-1}}+ \cdots +a_{r}D^{p^{r}}=0,
\end{equation}
with $a_{r}\neq 0$.
It is then not hard to see, as in~\cite[Section 1.5]{Strade:book}, or see
our Remark~\ref{rem:g} below, that there is a $p$-polynomial
$g(t)=\sum_{i=r}^{n-1}b_{i}T^{p^i}$
such that $g(D)^p-g(D)=D^{p^r}$.

\begin{theorem}\label{thm:general}
Let $A=\bigoplus A_k$ be a nonassociative algebra over the perfect field
$\F$ of prime characteristic $p$, graded over the integers modulo $m$.
Suppose that $A$ has a graded derivation $D$ of degree $d$ with
$m\mid pd$, such that $D^{p^r}$ is diagonalizable over $\F$.
Suppose that there exists a $p$-polynomial $g(T)\in \F[T]$ such that
$g(D)^p-g(D)=D^{p^r}$.
Set $h(T)=\sum_{i=1}^{r-1}T^{p^{i}}\in\F_p[T]$.

Let $A=\bigoplus_{\rho\in\F}A^{(\rho)}$ be the decomposition of $A$
into a direct sum of generalized eigenspaces for $D$
(with $A^{(\rho)}$ corresponding to the eigenvalue $\rho$).
Let $\LL_D:A\to A$ be the linear map whose restriction to
$A^{(\rho)}$ coincides with $L_{p-1}^{(g(\rho)-h(D))}(D)$.
Then
$A=\bigoplus_k \LL_D(A_k)$ is also a grading of $A$ over the integers modulo $m$.
\end{theorem}

\begin{proof}
We adapt the proof of Theorem~\ref{thm:special} to the present more
general setting.
Note that $h(T^p-T)=h(T)^p-h(T)=T^{p^r}-T^p$,
and that in the special case of Theorem~\ref{thm:special} we had
$h(T)=0$ and $g(T)=\gamma T^p$, where $\gamma^p-\gamma-1=0$.

The linear map $\LL_D$ is bijective.
In fact,
because $g(\rho)^p-g(\rho)=\rho^{p^r}$ for any eigenvalue of $D$,
on the generalized eigenspace $A^{(\rho)}$ of $D$
the linear map
$\bigl(g(\rho)-h(D)\bigr)^{p^r}
=g(\rho)^{p^r}-h(D^{p^r})$
acts as multiplication by the scalar
\[
g(\rho)^{p^r}-h(\rho^{p^r})
=g(\rho)^{p^r}-h\bigl(g(\rho)^p-g(\rho)\bigr)
=g(\rho)^p,
\]
and hence $(\LL_D)^{p^r}$ acts on $A^{(\rho)}$
as multiplication by the scalar
\[
\bigl(L_{p-1}^{(g(\rho)-h(\rho))}(\rho)\bigr)^{p^r}
=
L_{p-1}^{(g(\rho)^{p^r}-h(\rho^{p^r}))}(\rho^{p^r})
=
L_{p-1}^{(g(\rho)^p)}\bigl(g(\rho)^p-g(\rho)\bigr),
\]
which is nonzero according to Lemma~\ref{lemma:Laguerre}.
Hence we have the direct sum decomposition $A=\bigoplus_k \LL_D(A_k)$.

As in the proof of Theorem~\ref{thm:special}
we consider the multiplication map $m: A\otimes A \rightarrow A$,
and note that because $D$ is a derivation, and hence $D\circ m=m \circ(D\otimes\Id+\Id\otimes D)$, we have
\[
L_{p-1}^{(\theta(D))}(D)\circ m=m \circ L_{p-1}^{(\theta(D\otimes\Id+\Id\otimes D))}(D\otimes\Id+\Id\otimes D)
\]
for any polynomial $\theta(t)\in\F[t]$.
Fix eigenvalues $\rho,\sigma\in\F$ for $D$, view $m$ as restricted to a map
$A^{(\rho)}\otimes A^{(\sigma)}\rightarrow A^{(\rho+\sigma)}$,
and view the commuting linear operators $D\otimes\Id$ and $\Id\otimes D$ as restricted to
$A^{(\rho)}\otimes A^{(\sigma)}$.

We intend to evaluate the congruence of Proposition~\ref{prop:equation_for_L}
on $D\otimes\Id$ and $\Id\otimes D$ for $X$ and $Y$,
with $g(\rho)-h(D\otimes\Id)$ and $g(\sigma)-h(\Id\otimes D)$ for $\alpha$ and $\beta$.
To see that this makes sense we first need to check that the denominators
of the rational expressions
$c_i(\alpha,\beta)$ appearing in the congruence of Proposition~\ref{prop:equation_for_L}
evaluate to invertible linear maps on
$A^{(\rho)}\otimes A^{(\sigma)}$.
In fact,
$\alpha^{p^r}$ evaluates to
$\bigl(g(\rho)-h(D)\bigr)^{p^r}\otimes\Id$,
which, from what we saw earlier in the proof, acts on $A^{(\rho)}\otimes A^{(\sigma)}$
as scalar multiplication by
$g(\rho)^p$.
Together with the analogous fact about $\beta^{p^r}$, this shows that both
$(\alpha+\beta)^{p^r}$ and
$\bigl((\alpha+\beta)^p-(\alpha+\beta)\bigr)^{p^r}
=\bigl((\alpha+\beta)^{p-1}-1\bigr)^{p^r}(\alpha+\beta)^{p^r}$
evaluate to linear maps acting scalarly on
$A^{(\rho)}\otimes A^{(\sigma)}$,
by the scalars
$g(\rho)^p+g(\sigma)^p$
and
$\bigl(g(\rho)+g(\sigma)\bigr)^{p^2}-\bigl(g(\rho)+g(\sigma)\bigr)^p
=(\rho+\sigma)^{p^{r+1}}$,
respectively.
Whether $\rho+\sigma$ vanishes or not,
it follows that
$(\alpha+\beta)^{p-1}-1$
evaluates to an invertible linear map on
$A^{(\rho)}\otimes A^{(\sigma)}$,
in the former case because $\alpha+\beta$ evaluates to a nilpotent map.

All this can be stated more formally by saying that
evaluating both sides of the congruence of Proposition~\ref{prop:equation_for_L} amounts to apply a ring homomorphism from
$\F_p\bigl[\alpha,\beta,\bigl((\alpha+\beta)^{p-1}-1\bigr)^{-1},X,Y\bigr]$
to $\F[D]$, the subring of the ring $\End_{\F}(A^{(\rho)}\otimes A^{(\sigma)})$
of linear endomorphisms generated by $\F$ and $D$.
To ensure that such a homomorphism exists we have just checked
that $(\alpha+\beta)^{p-1}-1$ is mapped to an invertible element
of $\F[D]$.
In conclusion, both sides of the congruence of Proposition~\ref{prop:equation_for_L}
can be evaluated as described.
However, to draw any conclusion from this evaluation we need to
make sure that the ideal of the ring
$\F_p\bigl[\alpha,\beta,\bigl((\alpha+\beta)^{p-1}-1\bigr)^{-1},X,Y\bigr]$ generated by
$X^p-(\alpha^p-\alpha)$
and
$Y^p-(\beta^p-\beta)$
evaluates to zero.
This is so because both generators evaluate to zero.
In fact, the former evaluates to
\begin{multline*}
(D\otimes\Id)^p-
\Bigl(
\bigl(g(\rho)-h(D\otimes\Id)\bigr)^p-\bigl(g(\rho)-h(D\otimes\Id)\bigr)
\Bigr)
\\=
(D\otimes\Id)^{p^r}-
\bigl(g(\rho)^p-g(\rho)\bigr)
=
\Bigl(
D^{p^r}-
\bigl(g(\rho)^p-g(\rho)\bigr)
\Bigr)
\otimes\Id,
\end{multline*}
which acts as zero on
$A^{(\rho)}\otimes A^{(\sigma)}$.

The result of evaluating the congruence of Proposition~\ref{prop:equation_for_L} as described,
followed by composition with the multiplication map $m$, is that
the restriction of
\[
m\circ
L_{p-1}^{(g(\rho)-h(D\otimes\Id))}(D\otimes\Id)
\circ
L_{p-1}^{(g(\sigma)-h(\Id\otimes D))}(\Id\otimes D)
=
m\circ
\bigl(\LL_D
\otimes
\LL_D\bigr)
\]
to
$A^{(\rho)}\otimes A^{(\sigma)}$
coincides with the restriction of
\[
L_{p-1}^{(g(\rho)+g(\sigma)-h(D))}(D)
\circ m\circ
\biggl(
c_0(\alpha_0,\beta_0)+\sum_{i=1}^{p-1}c_i(\alpha_0,\beta_0)(D^i\otimes D^{p-i})
\biggr),
\]
where we have set $\alpha_0=g(\rho)-h(D\otimes\Id)$
and $\beta_0=g(\sigma)-h(\Id\otimes D)$
for the sake of readability.
This means that
\begin{equation}\label{eq:L_D_general}
\begin{aligned}
\LL_Dx \cdot \LL_Dy
=
\LL_D
\biggl(&
c_0\bigl(g(\rho)-h(D),g(\sigma)-h(D)\bigr)xy
\\&
+\sum_{i=1}^{p-1}c_i\bigl(g(\rho)-h(D),g(\sigma)-h(D)\bigr)D^ix\cdot D^{p-i}y\biggr)
\end{aligned}
\end{equation}
for $x\in A^{(\rho)}$ and $y\in A^{(\sigma)}$.

Because $D^p$ is a graded derivation of degree zero, it maps each
component $A_k$ of the grading into itself, and hence
$A_k=\bigoplus_{\rho\in\F}A_k\cap A^{(\rho)}$.
Because $m\mid pd$, the term $D^ix\cdot D^{p-i}y$ in Equation~\eqref{eq:L_D_general},
for $x\in A_k\cap A^{(\rho)}$
and $y\in A_\ell\cap A^{(\sigma)}$,
belongs to $A_{k+\ell}$ as well as the term $xy$.
Furthermore, each of the linear maps
$c_i\bigl(g(\rho)-h(D),g(\sigma)-h(D)\bigr)$ on the space $A^{(\rho+\sigma)}$,
for $0\le i<p$,
can be written as a polynomial map in $D^p$,
and hence sends $A_{k+\ell}\cap A^{(\rho+\sigma)}$
into itself, again because $D^p$ is a derivation of degree zero.
Hence Equation~\eqref{eq:L_D_general} tells us that
$\LL_Dx\cdot \LL_Dy\in \LL_D\bigl(A_{k+\ell}\cap A^{(\rho+\sigma)}\bigr)$.
In particular, we conclude that
$\LL_D A_k\cdot \LL_D A_\ell\subseteq \LL_D A_{k+\ell}$,
and so
$A=\bigoplus_k \LL_D(A_k)$
is a grading of $A$ over the integers modulo $m$.
\end{proof}

\begin{rem}
Restricted to the subalgebra $\ker(D^{p^r})$, where $\rho=0$, the map $\LL(D)$
coincides with that obtained by applying a variation of the Artin-Hasse exponential,
namely, the series $S(X)$ considered in~\cite[Section~3]{Mat:Artin-Hasse},
to which we refer the reader for details.
\end{rem}

\begin{rem}\label{rem:g}
Following~\cite[Section 1.5]{Strade:book} we sketch the construction of a $p$-polynomial
$g(T)=\sum_{i=r}^{n-1}b_{i}T^{p^i}$
such that $g(D)^p-g(D)=D^{p^r}$.
One way is to introduce a parameter $\lambda$ and impose that
$g(T)^p-g(T)-T^{p^r}=\lambda^p\sum_{i=r}^{n}a_iT^{p^i}$.
Starting from $b_{n-1}=\lambda$,
the equation recursively determines $b_h$ in terms of $\lambda$ as
$b_{h}^p=b_{h+1}+\lambda ^p a_{h+1}$,
for $h=n-2,n-3,\ldots,r$, and also forces
$-1-b_r=\lambda^pa_{r}$.
Hence
$b_{h}=-1-\sum_{k=r}^{h}\lambda^{p^{h+1-r}}a_k^{p^{h-k}}$, for $h=r,\ldots, n-1$,
where $\lambda$ is chosen among the roots of the polynomial
$1+T+\sum_{k=r}^{n-1}T^{p^{n-k}}a_k^{p^{n-1-k}}$.
\end{rem}

\section{Toral switching in restricted Lie algebras}\label{sec:toral_switching}

In this final section we discuss the connection with the \emph{toral switching} in modular Lie algebras.
Roughly speaking, this technique replaces a torus $T$ of a restricted Lie algebra $L$
with another torus $T_x$ which is more suitable for further study of $L$.
In the simplest and original setting of~\cite{Win:toral} this amounts to applying to $T$ the
exponential of the inner derivation $\ad x$, for some root vector $x\in L$ with respect to $T$.
Because $(\ad x)^2T=0$ the exponential of $\ad x$ can be taken to
be $1+\ad x$ for this purpose.
This is reminiscent of, and certainly motivated by,
the classical characteristic
zero situation where $\exp(\ad x)$ for some root vector $x$ is used to
conjugate a Cartan subalgebra into another.
However, in more general settings $(1+\ad x)T$ fails to be a
torus, and hence the construction of $T_x$ is slightly more involved.
This technique was originally introduced by Winter in~\cite{Win:toral} and later generalized by Block and Wilson
in~\cite{BlWil:rank-two}.
The most general version was finally produced by Premet
in~\cite{Premet:Cartan}.
An exposition of Premet's version can be found in~\cite[Section~1.5]{Strade:book}.

A crucial step in this process is to keep track of the root space decomposition with respect to the new torus,
by constructing linear maps from the root spaces with respect to $T$ onto the root spaces with respect to $T_x$.
Following Strade's exposition in~\cite[Section~1.5]{Strade:book}
we briefly sketch the construction of the new torus $T_x$ and of a linear map $E(x,\lambda)$
which connects the old and new root spaces.
Our goal is to show that $E(x,\lambda)$
coincides with the map $\LL_D$ of Theorem \ref{thm:general}, where $D=\ad x$.
This shows that the toral switching process, if we disregard the strictly Lie-theoretic aspects,
can be viewed as a special instance of Theorem \ref{thm:general}.
We only include enough details and notation to make the specialization of our results to toral switching readable
in conjunction with~\cite[Section~1.5]{Strade:book},
and refer to that source for more.

Let $L$ be a finite-dimensional restricted Lie algebra, over a perfect field $\F$ of positive chacteristic $p$, with $p$-mapping $[p]$.
Let $r$ be the difference between $\dim(L)$
and the maximum dimension of a torus of $L$
(but any larger integer would do).
In particular, $x^{[p^r]}$ is semisimple for each $x\in L$.
It is shown in~\cite[Section~1.5]{Strade:book} how to associate to each element $x$ of $L$ a
certain element $\xi(x,\lambda)$ of $L$,
which also depends on a choice of a certain admissible scalar
$\lambda\in\F$, itself depending on $x$.
This is done in a systematic `polynomial' way whose details we disregard here,
except for pointing out that
when $D=\ad x$ the map $\ad\xi(x,\lambda)$ plays the role of our $g(D)$
in the previous section.
The crucial property of $\xi(x,\lambda)$ is that
\begin{equation}\label{eq:xi}
\xi(x,\lambda)^{[p]}-\xi(x,\lambda)=x^{[p]^r}.
\end{equation}
Set $q(x)=\sum_{t=1}^{r-1}x^{[p]^t}$.
Strade then defines the map $E_{(x,\lambda)}$ as
\[
E_{(x,\lambda)}=
-\sum_{i=0}^{p-1}\biggl(\prod_{k=i+1}^{p-1}\bigl(\ad\xi(x,\lambda)-\ad q(x)+kId\bigr)\biggr)
(\ad x)^{i}.
\]

Now let $T$ be a torus of $L$ of maximal dimension, and let $L=\bigoplus_{\gamma \in \Gamma} L_{\gamma}$ be the corresponding root space decomposition
(where $\Gamma=\Gamma(L,T)$ in~\cite{Strade:book}).
Let $x \in L_{\beta}$ be a root vector (hence with $\beta \neq 0$) such that $x^{[p]^r}\in T$, whence each $L_{\gamma}$ is an eigenspace for
$\ad x^{[p]^r}=(\ad x)^{p^r}$.
It is stated in~\cite[Theorem 1.5.1]{Strade:book} that
$T_x=\{t-\beta(t)(\sum_{k=0}^{r-1}x^{[p]^k}): t \in T\}$
is also a torus of $L$, and that
$L=\bigoplus_{\gamma \in \Gamma} E_{(x,\lambda)}L_{\gamma}$ is the corresponding root space decomposition.
Note that
$(1+\ad x)t=t-\beta(t)x$ can be taken as an interpretation of $\exp(\ad x)t$
because $(\ad x)^2t=0$;
however, the elements used to define $T_x$ above are more complicated than that, in general.

We now show that the map $E_{(x,\lambda)}$
coincides with the map $\LL_D$ of our Theorem~\ref{thm:general}.
Setting $D=\ad x$ in the situation of~\cite[Theorem~1.5.1]{Strade:book} we have that $D^{p^r}$ is semisimple.
The polynomial $g$ of Section~\ref{sec:general_case} was defined in such a way that $g(D)=\ad\xi(x,\lambda)$, and obviously $h(D)=\ad q(x)$.
We know that $g(D)$ acts scalarly on any generalized eigenspace for $D$
(which for $\ad\xi(x,\lambda)$ can be deduced from
Equation~\eqref{eq:xi}).
Now any $L_{\gamma}$ (a root space, or $L_0$) is contained in the generalized eigenspace for $D$ with respect to the eigenvalue $\rho$, where $\rho\in\F$
is determined by
$\gamma(x^{[p^r]})=\rho^{p^r}$.
Then the map $E_{(x,\lambda)}$ acts on $L_{\gamma}$ as
\[
E_{(D,\lambda)}=-\sum_{i=0}^{p-1}\biggl(\prod_{k=i+1}^{p-1}\bigl(g(\rho)-h(D)+k\Id\bigr)\biggr)
D^{i}=L_{p-1}^{(g(\rho)-h(D))}(D),
\]
and therefore coincides with our $\LL_D$.

To conclude our comparison with toral switching we show that part of the information
given in~\cite[Theorem~1.5.1]{Strade:book}, namely, that
$\bigoplus_{\gamma\in\Gamma}E_{(x,\lambda)}L_{\gamma}$
is a grading of $L$ (over $\langle\Gamma\rangle$, the additive group generated by $\Gamma$), is a consequence of our
Theorem~\ref{thm:general}.
Of course, a crucial part of the toral
switching technique is that this grading is actually the root space decomposition of a new torus $T_x$,
but this part loses meaning in our more general setting where $A$ is an arbitrary nonassociative algebra.

In loose terms, toral switching modifies the original
grading (that is, root space decomposition) in only one direction and does not affect it in suitably
complementary directions.
Our formulation of Theorem~\ref{thm:general} for a cyclic grading
means that it focuses on the one `direction' where the switching
takes place, and so we need a little work to
isolate that direction before Theorem~\ref{thm:general} becomes
applicable to the toral switching setting.

Because $\beta(t^{[p]})=\beta(t)^p$ for $t\in T$
(see~\cite[Equation~(1.3.2)]{Strade:book}), the maximal subspace
$T_0:=\ker(\beta)=\{t\in T:\beta(t)=0\}$
is a $p$-subalgebra of the torus $T$, and hence a torus itself.
The restriction $\gamma\mapsto\gamma_{T_0}$ to $T_0$ gives a surjective $\F$-linear map
$T^\ast\to T_0^\ast$, with kernel spanned by $\beta$.
Let $\Gamma_0$ be the image of $\Gamma$ under this restriction map.
Note that the subgroup $\langle\Gamma_0\rangle$ generated by $\Gamma_0$
has rank one less than the rank of $\langle\Gamma\rangle$.
Choose a toral element $t_1\in T$ (hence $t_1^{[p]}=t_1$) with $\beta(t_1)=1$;
this can be done because $T$ is spanned by toral elements.
We have a group isomorphism of $\langle\Gamma\rangle$ with the direct product
$\F_p\times\langle\Gamma_0\rangle$, where to $\gamma$ there corresponds the
pair $\bigl(\gamma(t_1),\gamma_{T_0}\bigr)$.

For $\gamma_0\in\Gamma_0$
the sum
$L_{\gamma_0}:=\bigoplus_{\gamma\in\Gamma\colon\gamma_{T_0}=\gamma_0}L_{\gamma}$
is a root space for the torus $T_0$.
Hence
$L=\bigoplus_{\gamma_0\in\Gamma_0}L_{\gamma_0}$
is the root space decomposition of $L$ with respect to $T_0$.
Similarly,
$L=\bigoplus_{k\in\F_p}L_k$,
where
$L_k:=\bigoplus_{\gamma\in\Gamma\colon\gamma(t_1)=k}L_{\gamma}$,
is the root space decomposition of $L$ with respect to the torus
spanned by $t_1$.
The root space decomposition of $L$ with respect to $T$ can be
viewed as a grading
\[
L=\bigoplus_{(k,\gamma_0)\in\F_p\times\langle\Gamma_0\rangle}L_k\cap L_{\gamma_0}
\]
over $\F_p\times\langle\Gamma_0\rangle$.
Now our Theorem~\ref{thm:general} applies, with $m=p$, to the grading
$L=\bigoplus_{k\in\F_p}L_k$,
and yields a grading
$L=\bigoplus_{k\in\F_p}\LL_DL_k$.
(Following~\cite{Strade:book} one may show that this is the root space decomposition with respect to the torus spanned by $t_1-x-h(x)$,
but we may ignore this fact here.)

Because $[T_0,x]=0$, the derivation $D=\ad x$ commutes with $\ad t$ for each $t\in T_0$.
Consequently, each $L_{\gamma_0}$ is invariant under the linear map $\LL_D$, because the latter can be expressed as a
polynomial in $D$ on $L_{\gamma_0}$.
Therefore, $\LL_DL_{\gamma_0}=L_{\gamma_0}$ for each $\gamma_0\in\Gamma_0$, being
$\LL_D$ bijective, and hence
$\LL_D(L_k\cap L_{\gamma_0})=\LL_DL_k\cap L_{\gamma_0}$
for
$(k,\gamma_0)\in\F_p\times\langle\Gamma_0\rangle$.
Because both of
$L=\bigoplus_{k\in\F_p}\LL_DL_k$
and
$L=\bigoplus_{\gamma_0\in\Gamma_0}L_{\gamma_0}$
are gradings (according to Theorem~\ref{thm:general} in case of the
former), the direct sum decomposition
\[
L=
\bigoplus_{(k,\gamma_0)\in\F_p\times\langle\Gamma_0\rangle}
\LL_D(L_k\cap L_{\gamma_0})
=
\bigoplus_{(k,\gamma_0)\in\F_p\times\langle\Gamma_0\rangle}
\LL_DL_k\cap L_{\gamma_0}
\]
is a grading as well.
This is equivalent to saying that
$\bigoplus_{\gamma\in\Gamma}\LL_DL_{\gamma}$
is a grading of $L$, as we wanted to prove.

\bibliography{References}
\end{document}